\documentclass[11pt]{amsart}
\usepackage{setspace}
\usepackage{amsmath}
\usepackage{hyperref}
\usepackage{graphicx}
\usepackage{amssymb,amsthm}
\usepackage{hyperref}
\usepackage{esint}
\usepackage{epstopdf}
\usepackage{color}
\usepackage{url}
\usepackage[top=1in, left=1in, right=1in, bottom=1in, marginpar=2cm]{geometry}

\newtheorem{theorem}{Theorem}[section]
\newtheorem{lemma}[theorem]{Lemma}
\newtheorem{proposition}[theorem]{Proposition}
\newtheorem{corollary}[theorem]{Corollary}

\newtheorem{remark}[theorem]{Remark}

\newtheorem{definition}{Definition}[section]

\def\Z{{\mathbb Z}}
\def\R{{\mathbb R}}

\def\B{{\mathcal B}}
\def\S{{\mathcal S}}

\def\cB{{\mathcal B}}

\def\h1{\> d{\mathcal H}^1}

\def\b{\beta}
\def\e{\varepsilon}
\def\d{\delta}

\def\p{\partial}

\def\1{\left(}
\def\2{\right)}
\def\3{\left\{}
\def\4{\right\}}
\def\8{\infty}
\def\sm{\setminus}
\def\ss{\subseteq}
\def\cc{\subset\subset}

\DeclareMathOperator*{\dvg}{div}

\usepackage{enumerate}

\def\T{{\mathbb T}}

\newtheorem{question}[theorem]{Question}

\begin{document}

\title{A min-max variational approach to the existence of gravity water waves}

\author{Dennis Kriventsov}
\address[Dennis Kriventsov]{Rutgers University, Piscataway, NJ}
\email{dnk34@math.rutgers.edu}

\author{Georg S. Weiss}
\address[Georg S. Weiss]{Faculty of Mathematics, University of Duisburg-Essen, Germany}
\email{georg.weiss@uni-due.de}

\begin{abstract}
We establish the existence of gravity water waves by applying a mountain pass theorem to a singular perturbation of the Alt-Caffarelli functional associated with the two-dimensional water wave equations. Our approach is formulated entirely in physical coordinates and does not require the air phase to be connected, nor does it rely on symmetry or monotonicity in the $x$ or $y$ directions. The framework presented allows for both a variational approach to a variety of fluid equilibrium problems and for construction of min-max solutions to Bernoulli-type free boundary problems.
\end{abstract}

\maketitle

\section{Introduction}
We consider a two-dimensional inviscid and incompressible fluid influenced by gravity and possessing a free surface. Let $D(t) \subset \mathbb{R}^2$ denote the region occupied by the fluid at time $t$. The fluid dynamics are governed by the Euler equations for the velocity field $(u(t, \cdot), v(t, \cdot)) : D(t) \rightarrow \mathbb{R}^2$ and the pressure field $P(t, \cdot) : D(t) \rightarrow \mathbb{R}$:
\begin{equation*}
	\begin{cases}
    u_t + u u_x + v u_y = -P_x & \text{in } D(t), \\
    v_t + u v_x + v v_y = -P_y - g & \text{in } D(t), \\
    u_x + v_y = 0 & \text{in } D(t),
	\end{cases}
\end{equation*}
where subscripts represent partial derivatives, and $g$ denotes the gravitational constant. The boundary $\partial D(t)$ includes a free surface segment, denoted $\partial_a D(t)$, which is in contact with the surrounding air. The equations are supplemented with the standard boundary conditions:
\begin{equation*}
	\begin{cases}
    V = (u, v) \cdot \nu \quad \text{on } \partial_a D(t), \\
    P \quad \text{is locally constant on } \partial_a D(t),
	\end{cases}
\end{equation*}
where $V$ represents the normal velocity of the free surface $\partial_a D(t)$ and $\nu$ is the outward unit normal vector. Additionally, we assume the flow is irrotational:
\[
    u_y - v_x = 0 \quad \text{in } D(t).
\]

Focusing on traveling wave solutions, we consider a fixed domain $D \subset \mathbb{R}^2$, a speed $c \in \mathbb{R}$, and functions $(\tilde{u}, \tilde{v}) : D \rightarrow \mathbb{R}^2$ and $\tilde{P} : D \rightarrow \mathbb{R}$ such that
\begin{equation*}
	\begin{cases}
    D(t) = D + ct(1, 0) & \text{for all } t \in \mathbb{R}\\
    u(x, y, t) = \tilde{u}(x - ct, y) + c & \\
	v(x, y, t) = \tilde{v}(x - ct, y) & \\
	P(x, y, t) = \tilde{P}(x - ct, y). &
	\end{cases}
\end{equation*}
This leads to the following steady-state equations in $D$:
\begin{equation*}
	\begin{cases}
    \tilde{u} \tilde{u}_x + \tilde{v} \tilde{u}_y = -\tilde{P}_x & \\
    \tilde{u} \tilde{v}_x + \tilde{v} \tilde{v}_y = -\tilde{P}_y - g & \\
    \tilde{u}_x + \tilde{v}_y = 0 & \\
    \tilde{u}_y - \tilde{v}_x = 0 & \\
    (\tilde{u}, \tilde{v}) \cdot \nu = 0 & \text{on } \partial_a D \\
    \tilde{P} \text{is locally constant} & \text{on }  \partial_a D.
\end{cases}
\end{equation*}
This framework captures both water waves (with homogeneous Neumann conditions at a flat bottom $y = -d$ and periodic or other conditions at $x = \pm \infty$) and fluid equilibrium problems with lateral inflow and outflow in a bounded domain with possibly non-flat bottom boundaries.

In both scenarios, incompressibility and the kinematic boundary condition imply the existence of a stream function $\psi$ in $D$, up to a constant, defined by:
\[
    \psi_x = -\tilde{v}, \quad \psi_y = \tilde{u}.
\]
Hence, $\psi$ is locally constant  on $\partial_a D$. In the water wave case, $\psi$ is also locally constant on the flat bottom. Irrotationality implies that $\psi$ is harmonic in $D$, i.e.,
\[
    \Delta \psi = 0 \quad \text{in } D.
\]
Bernoulli's principle then gives
\[
    \tilde{P} + \frac{1}{2} |\nabla \psi|^2 + g y = \text{constant in } D,
\]
and the dynamic boundary condition yields the Bernoulli condition
\[
    |\nabla \psi|^2 + 2 g y = \text{locally constant on } \partial_a D.
\]
In this paper, we focus on the case of periodic water waves with finite depth, which, after normalization, is given by
\begin{equation}\label{eq:intropde}
	\begin{cases}
    \Delta \psi = 0 & \text{in } \left(\T \times [0, \infty)\right) \cap \{\psi>0\}, \\
    |\nabla \psi(x, y)|^2 = A - B y & \text{on } \left(\T \times [0, \infty)\right) \cap \partial \{\psi>0\},\\
    \psi(\cdot,0)=1 & \text{on } \T.
	\end{cases}
\end{equation}
Here and below, we use the notation $\T := \R/\Z$, identifying functions on $\T$ 
satisfying periodic lateral boundary conditions with $1$-periodic functions on $\R$ and
assume that $\psi$ is $1$-periodic in $x$. 
The variables $x$ and $y$ below will be located in $(x, y) \in \T \times [0, \infty)$.

Existence results for large-amplitude smooth waves have been obtained by Krasovskii \cite{k}, and by Keady and Norbury \cite{kn}. The existence of large-amplitude smooth solitary waves and of extreme solitary waves has been shown by Amick and Toland \cite{at}.
All of these existence results, as well as many subsequent works, use an equivalent formulation of the problem as a non-linear singular integral equation due to Nekrasov (derived via conformal mapping).

Another approach to finding non-trivial solutions of the fluid equilibrium problems with lateral inflow and outflow in a bounded domain 
is to minimize the Alt-Caffarelli energy with a gravity term
\begin{equation}\label{eq:introenergy}
	E[\psi] = \int_{\T \times [0, \infty)} \left( |\nabla \psi|^2 + \chi_{\{\psi > 0\}} (A - B y)_+\right).
\end{equation}
The Euler-Lagrange equation for this functional is precisely \eqref{eq:intropde}. However, naive minimization of this energy with boundary condition $\psi(x, 0) = c > 0$ will only lead to the trivial flat wave. This was observed and studied in \cite{aramaleoni}, where the authors then also study minimizers with non-constant boundary conditions and other configurations. In \cite{gravina-leoni}, a different approach is taken to get non-flat solutions: roughly speaking, $\psi$ is constrained to be $0$ along a line segment $\{1/2\} \times [l, \infty)$ in a way which precludes the flat wave from being a solution. The authors then study the behavior of the minimizers, including near the point $(1/2, l)$. It is not, however, clear that for some parameter $l$ the resulting constrained minimizer is truly a solution of \eqref{eq:intropde} at the point $(1/2, l)$.

Formally, a minimization with double constraint, that is 
minimization of $$\int_{\T \times [0, \infty)} |\nabla \psi|^2$$
with the constraints
$$ \int_{\T \times [0, \infty)} \chi_{\{\psi> 0\}} = c_1, \quad  \int_{\T \times [0, \infty)} y\chi_{\{\psi> 0\}} = c_2$$
should lead to nontrivial waves. However, along minimizing sequences for this problem, part of the volume will escape to infinity, leading to a loss of compactness in direct method arguments. This phenomenon also appears to occur in numerical simulations; the authors are grateful to Antoine Laurain and Josue Daniel Diaz Avalos for analyzing this formulation from a numerical standpoint. Even when working in a class of monotone-in-$y$ functions, similar loss of compactness (via cusps with vertical lateral boundary) is still present.

In light of these considerable difficulties--only some of which were known at the time--John Toland raised the following question (paraphrased) in a discussion with Eugen Varvaruca and the second author:
\begin{question}
	Can one obtain any (even small amplitude) existence results for \eqref{eq:intropde} by variational methods in the original variables?
\end{question}
The main goal of this paper is to show existence of large-amplitude smooth periodic waves via a mountain pass approach. Our approach is formulated entirely in physical coordinates and does not require the air phase to be connected.

Our main results are:
\begin{theorem}\label{main1}
Assume that $B < 2 (\frac{A}{3})^{3/2}$ and also 
\begin{align*}\label{cond2}
2 \frac{A}{B} 2 \pi & \left( 1 - \frac{1}{3}\left(1+2 \cos\left( \frac{1}{3} \arccos\left( 1-\frac{27B^{2}}{2A^{3}}\right) \right)\right)\right)\times\\
			&\coth\left(2 \pi \frac{1}{3} \frac{A}{B} \left(1+2 \cos\left( \frac{1}{3} \arccos\left( 1-\frac{27B^{2}}{2A^{3}}\right) \right)\right)\right) < 1.
			\end{align*}
Then there exists a domain variation critical point of $E$ (defined on $\T \times [0, \infty)$ and $1$-periodic in $x$) that is not independent of $x$.
\end{theorem}
The conditions on $A, B$ are discussed in greater detail in Remark \ref{aboutcondition2} and Figure \ref{fig:abplot}, and can be rewritten in other ways. They are precisely the set of parameters when there are two distinct flat waves, with one locally minimal while the other sufficiently unstable.

We do not, at this time, know the maximal amplitude of the waves produced by our approach,
and we do not even know whether our water waves are on the same branch as 
those in \cite{at}. Let us, however, emphasize that for our existence approach neither symmetry nor monotonicity in the $x$ or $y$ directions
are necessary. This may be of interest, as numerical results indicate the existence of non-symmetric waves
(\cite{chensaffman}, \cite{vanden-broeck}, \cite{zufiria})
as well as water waves non-monotone in the $y$-direction (\cite{constantinvarvaruca}, \cite{wahlen}).
We can, however, also produce waves with symmetries, which have enhanced regularity properties.

\begin{theorem}\label{main2}
Assume that $B < 2 (\frac{A}{3})^{3/2}$ and also 
\begin{align*}
2 \frac{A}{B} 2 \pi & \left( 1 - \frac{1}{3}\left(1+2 \cos\left( \frac{1}{3} \arccos\left( 1-\frac{27B^{2}}{2A^{3}}\right) \right)\right)\right)\times\\
			&\coth\left(2 \pi \frac{1}{3} \frac{A}{B} \left(1+2 \cos\left( \frac{1}{3} \arccos\left( 1-\frac{27B^{2}}{2A^{3}}\right) \right)\right)\right) < 1.
			\end{align*}
Then there exists a domain variation critical point $u$ of $E$ (defined on $\T \times [0, \infty)$ and $1$-periodic in $x$) that is not independent of $x$, $u(x, y) = u(-x, y)$, and $u$ is symmetrically decreasing, that is $u_x(x, y) \leq 0$ for $x \in (0, 1/2)$.
The free boundary $\partial \{ u>0\}$ is the graph of a function of $y$, that is,
	$\partial \{ u>0\} = \{ (f(y),y): y \in S\}$, where $S$ is a closed subset of $[0,A/B]$.
The water surface
$\S := \{ (f(y),y): y \in I\}$, where $I$ is the first/leftmost connected component of $S$
is regular in the sense that
$\S \setminus \left((0,A/B) \cup \{ |x| = 1/2 \}\right)$ is locally the graph of an analytic function.
Moreover, either $\S\setminus (0,A/B)$ is locally the graph of an analytic function,
or there is a downward-pointing cusp of $\S$ at $|x|=1/2$ at which non-$\S$ free boundary
points must exist that converge to the cusp point.
\end{theorem}

The basic idea of the proof is, in some sense, straightforward, but presents challenges in the execution. We begin by studying the energy structure of the Alt-Caffarelli functional \eqref{eq:introenergy} (as was, in fact, already done in \cite{aramaleoni}): for the values of $A, B$ under consideration, there are only three one-dimensional critical points, with two of them local minimizers and one being unstable. The key further observation we make is that, again for the parameters as above, the unstable solution has Morse index at least 2. Formally, then, one should be able to apply a mountain pass theorem to curves connecting the two local minimizers to obtain a critical point of Morse index at most 1, which is then not any of these three flat solutions.

The main issue with making this rigorous is that there is no mountain pass theorem available in the literature for functionals like \eqref{eq:introenergy}, which are not differentiable. If one attempts to use classic versions like \cite{AR73}, it will be impossible to verify the Palais-Smale condition. An analogy can be made with the minimal surface functional, where an extensive min-max theory has been developed (and is an area of active study), but is extremely non-trivial and requires somewhat different ideas from the traditional semilinear context. Bernoulli-type free boundary problems like \eqref{eq:intropde} often exhibit similar difficulties to minimal surfaces.

In this paper, we present an elementary approach to min-max arguments for Bernoulli problems. First, we regularize \eqref{eq:introenergy} to $E_\e$ by smoothing out $\chi_{\{\psi > 0\}}$ to a mollified $\mathcal{B}_\e(\psi)$. This is a classic strategy in free boundaries, and it is easy to see that e.g. $E_\e$ $\gamma$-converges to $E$. In particular, the energy landscape of $E_\e$ is similar to that of $E$. Unlike $E$, $E_\e$ is smooth, satisfies the assumptions of standard mountain pass theorems like \cite{AR73}, and we successfully find the critical points we wanted. Then we ``simply" take a limit of these critical points, to get a critical point of $E$ itself. This strategy is reminiscent of the Allen-Cahn approach to min-max for minimal surfaces, albeit with a different semilinear approximation.

The main problem with this strategy would be that it is not at all clear that a limit of critical points to $E_\e$ is actually a critical point to $E$. This was an open question in the literature for a long time, but in a recent work \cite{kw}, the authors have been able to prove exactly such a compactness result. Moreover, in the Bernoulli context it is not difficult to pass second (inner) variation to the limit as well, and so the limiting critical point has Morse index at most one (this is different from the situation with minimal surfaces). We would like to emphasize that up to this point, the method is extremely general and requires minimal a priori knowledge of qualitative structure or regularity.

To prove Theorem \ref{main2}, we first produce symmetric and monotone min-max solutions by performing a Steiner symmetrization on the min-max setup. Then we use free boundary arguments to obtain the regularity stated. As our goal here is to present this overall strategy and its application to the water waves problem \eqref{eq:intropde}, we do not attempt to obtain the strongest possible regularity results here. We intend to explore that point in future work.

The organization of the paper is as follows: in Section \ref{s:prelim}, we set up basic notation and terminology. The energy landscape of $E$ and $E_\e$ is studied in Section \ref{s:energy}, to set up for the application of the mountain pass theorem in Section \ref{s:mountainpass}. To then pass to the limit and prove the main theorems in Section \ref{s:existence}, we first prove uniform Lipschitz estimates in Section \ref{s:lip}. Finally, Section \ref{s:regularity} deals with the regularity of the water surface.

\section{Preliminaries}\label{s:prelim}

Set
\[
	E[u] = \int_{\T \times [0, \infty)} \left( |\nabla u|^2 + \chi_{\{u > 0\}} (A - B y)_+\right),
\]
where $A, B$ are positive parameters. Our goal is to find critical points of $E$ in the space
\[
	H := \{u \in \dot{W}^{1, 2}(\T \times [0, \infty)) : u(x, 0) = 1\};
\]
here $\dot{W}^{1, 2}(\T \times [0, \infty))$ is the closure of $C^\infty_c(\T \times [0, \infty))$ with respect to the seminorm
$$\Vert\nabla u\Vert_{L^2(\T \times [0, \infty))} + \inf_{c\in \R}\Vert u-c\Vert_{L^2(\T \times [0, \infty))}.$$
Notice that there are three free parameters in this variational problem: the values $A, B$ and the value of $u$ along $\T \times \{0\}$. There is also one elementary scaling property available (multiplying $u$ by a constant) which we have used to normalize to $u = 1$ along $\T \times \{0\}$, leaving us with a two-parameter family.

We will first work with a ``regularized'' version of this energy. Fix $\cB : \R \rightarrow [0, \infty)$ be a smooth, nondecreasing function with $\cB(0) = 0$, $\cB(t) > 0$ for $t > 0$, and $\cB(t) = 1/2$ for $t \geq 1$. Then set $\cB_\e(t) = \cB(t/\e)$, and
\[
	E_\e[u] = \int_{\T \times [0, \infty)} \left(|\nabla u|^2 + 2 \cB_\e(u) (A - B y)_+\right).
\]
Formally, as $\e \rightarrow 0$, $E_\e \rightarrow E$.

For $E_\e$, there is a straightforward notion of critical point, which we will use below:
\begin{definition}[Outer variation critical point]
	A function $u \in H$ is an \emph{outer variation critical point} of $E_\e$ if for any $T > 0$ and $v \in W^{1, 2}_0(\T \times (0, T))$,
	\[
		\partial_t E_\e[u + t v]  |_{t = 0} = 0.
	\]
\end{definition}
This concept makes sense for $E$ as well, of course, but we will not be able to directly construct critical points of this type, and so will use a different notion instead:
\begin{definition}[Inner variation critical point]
	A function $u \in H$ is an \emph{inner variation critical point} of $E$ if:
	\begin{enumerate}
		\item $u$ is locally Lipschitz continuous on $\T \times [0, \infty)$.
		\item $u$ is harmonic on the (open) set $\{u \neq 0\} \cap \T \times (0, \infty)$.
		\item For any vector field $V \in C_c^\infty(\T \times (0, \infty))$ with flow $\phi_t$,
		\[
			\partial_tE[u \circ \phi_t^{-1}] |_{t = 0}  = 0.
		\]
	\end{enumerate}
\end{definition}
Assumption (2) is equivalent to asking that $u$ be an outer variation critical point to $E$ on the set $\{u \neq 0\}$, and does not follow directly from (3). In general (3) is weaker than the outer variation property in this context, but is better behaved under limits.

While in principle $H$ contains functions which may change sign, this is purely for ease of formulation, as we now check:

\begin{lemma}[Boundedness of outer variation critical points]\label{lem:bounded}
	Let $u$ be an outer variation critical point of $E_\e$. Then $u$ is $C^2$, $0 < u(x, y) \leq 1$ for $y > 0$, and $u$ solves
	\[
		\Delta u = \b_\e(u)(A - B y)_+
	\]
	on $\T \times (0, \infty)$.
\end{lemma}

\begin{proof}
	For any $v \in W^{1, 2}_0(\T \times [0, T])$, we have that
	\begin{equation}\label{weak_epsilon}
		\int \left(\nabla u \cdot \nabla v + \b_\e(u) (A - B y)_+ v\right) = 0
	\end{equation}
	from the outer variation condition. As $\b_\e$ is bounded, Schauder estimates give that $u \in C^{1, \alpha}$, and then in $C^{2, \alpha}$ for any $\alpha \in (0, 1)$; this also immediately gives the strong form of the PDE above.

	Notice that $u$ is harmonic for $y > A/B$. Any harmonic function on such a half-cylinder which has $\nabla u \in L^2$ must be bounded: $\sup_{\T \times [0, \infty)}|u| \leq C$. Indeed,
	\[
		\frac{d}{dy}\int_{\T} u(x, y) dx = \int_{\T} u_y(x, y) dx = \int_{\T} u_y(x, A/B)dx \text{ for } y>A/B
	\]
	from the divergence theorem. If $m = \int_{\T} u_y(x, A/B) dx \neq 0$, then
	\[
		|m| = \left|\int_{\T} u(x, k + 1) dx - \int_{\T} u(x, k) dx\right| \leq \sqrt{\int_{\T \times (k, k+1)} |u_y|^2 },
	\]
	and squaring and summing in $k$ violates the fact that $\nabla u \in L^2$. Therefore $q = \int_{\T} u(x, y) dx$ is constant, and so the Poincar\'e inequality gives
	\[
		\int_{\T \times (y, y+1)}|u - q|^2 \leq \int_{\T \times (y, y+1)} |\nabla u|^2 \leq C \text{ for } y\geq A/B.
	\]
	Then
	\[
		\sup_{\T \times [y+1/4, y+3/4]} |u| \leq C
	\]
	from the mean value property, for any $y > A/B$. On the compact region $\T \times \{y \leq A/B + 1/2\}$ we already have that $u$ is bounded from the $C^2$ estimate above.
	
	Let $v = (-u)_+$ or $(u - 1)_+$. Noticing that $v(x,0)=0$ and using $h = \eta^2 v$ as a test function in \eqref{weak_epsilon}, where $\eta = \eta(y)$ is a smooth decreasing cutoff function which is $1$ for $y \leq R$, $0$ for $y \geq R + 1$, and has $|\eta'|\leq 2$, we have the standard energy identity
	\[
		\int \eta^2 |\nabla v|^2 = -2 \int v \eta \nabla v \cdot \nabla \eta \leq C \sqrt{\int_{\T \times [R, R+1]} |\nabla v|^2 }
	\]
	using the boundedness of $u$. However, $|\nabla v|\leq |\nabla u|$, and $\nabla u\in L^2$, so the right-hand side must go to $0$ as $R\rightarrow \infty$. From monotone convergence on the left, then,
	\[
		\int |\nabla v|^2 = 0,
	\]
	and $0 \leq u \leq 1$.

	From the assumption that $\cB(0) = 0$ and $\cB$ is smooth, it is possible to write $\b_\e(t) = f(t)t$ for some smooth non-negative $f$, at which point $u(x, y) = 0$ implies $u \equiv 0$ by the strong maximum principle, contradicting that $u(x, 0) = 1$. Likewise $u(x, y) = 1$ implies $u \equiv 1$, which is an outer variation critical point.
\end{proof}

\begin{lemma}[Free boundary condition] \label{lem:FB}
	Let $u$ be an inner variation critical point of $E$. Then  $0 \leq u(x, y) \leq 1$ for $y > 0$. At a point $(x,y) \in \p \{u > 0\}$ where $\{u > 0\}$ is locally a smooth domain (i.e. it lies to one side of a smooth curve), we have that
	\[
		|\nabla u(x, y)|^2 = (A - B y)_+,
	\]
	where the derivative is understood to be from inside the domain.
\end{lemma}

We will discuss the question of whether $\p \{u > 0\}$ is smooth later, in Section \ref{s:regularity}. We will not usually need this strong form of the free boundary condition.

\begin{proof}
	That $0 \leq u \leq 1$ follows as for the semilinear case using the global weak maximum principle derived there. The free boundary condition at smooth points can be derived by performing small normal variations and is well-known, see \cite{ac}.
\end{proof}

Finally, we make a remark about the use of the compactness theorem in \cite{kw} below.

\begin{remark}\label{r:compactness} In the proofs of our main theorems below, we will need to invoke the main result of \cite{kw}, Theorem 1.2, but applied to the energy $E$. This energy contains a weight dependent on $y$, and so this constitutes a generalization of \cite{kw}. Such a generalization is valid, the central point being that the frequency formula used is actually true with straightforward modifications for $E$; this was already observed in \cite[Section 7]{vw}. However, a detailed treatment of such a generalization is not currently available in the literature. In a forthcoming work with Mark Vaysiberg, we will generalize \cite[Theorem 1.2]{kw} to a much wider class of functionals, including $E$. That paper is currently in preparation, and once it is available this remark will be updated with the reference. Usage of \cite[Theorem 1.2]{kw} as applied to $E$ in the text will cross-reference this remark.
\end{remark}

\section{The energy landscape}\label{s:energy}

In this section, we study the energy landscape of $E$ and $E_\e$ with the aim of finding a usable mountain-pass configuration. The limit problem is simpler in this regard, so we start by working with it directly before proceeding to the semilinear approximation. 

\subsection{One-dimensional critical points and minimizers of the energy}

We begin by classifying all minimizers of $E$ and all one-dimensional solutions.
The following two lemmas are similar to the analysis in \cite[Section 5]{aramaleoni}, with different notation and treatment of ``nonphysical'' solutions like $U_\infty$ below. We present the arguments in full for clarity.

\begin{lemma}[One-dimensional solutions]\label{lem:1d}
	Let $u \in H$ be an inner variation critical point of $E$, and assume $u$ is independent of $x$. Then:
	\begin{itemize}
		\item If $B < 2 (\frac{A}{3})^{3/2}$, then $u$ is one of the following three functions:
		\[
		\begin{cases}
			&U_\infty(y) = 1 \\
			&U_+(y) = (1 - y/Y_+)_+ \\
			&U_-(y) = (1 - y/Y_-)_+,
		\end{cases}
		\]
		where $0 < Y_- < \frac{2}{3}\frac{A}{B} < Y_+ < \frac{A}{B}$. The functions $U_-, U_\infty$ are local minimizers of $E$ (with respect to $1D$ variations), while $U_+$ is not a minimizer.
		\item If $B = 2 (\frac{A}{3})^{3/2}$, then the only possibilities for $u$ are $U_\infty$ and $U_0 = (1 - y/Y_0)_+$, where $Y_0 = \frac{2}{3}\frac{A}{B}$. $U_0$ is not a minimizer.
		\item If $B > 2 (\frac{A}{3})^{3/2}$, then $u = U_\infty$ (there are no other critical points).
	\end{itemize}
\end{lemma}

\begin{proof}
	As $u$ is harmonic where positive, it has a very simple structure: it is either the positive part of a linear function of $y$, $u(y)= u_Y(y) = (1 - y/Y)_+$, or $u=U_\infty$. Here $Y > 0$, for otherwise this function does not have $(u_Y)_y \in L^2$. The constant function $U_\infty$ is always a solution, and has $E[U_\infty] = \frac{A^2}{2B}$. For the others, the inner variation condition is equivalent to verifying that at the single free boundary point $y = Y$, we have that
	\[
		\frac{1}{Y^2} = |(u_Y)_y(Y)|^2 = (A - B Y)_+,
	\]
	so $Y < \frac{A}{B}$ and $p(Y) := A Y^2 - B Y^3 - 1 = 0$. Note also that
	\[
		e(Y) := E[u_Y] = \frac{1}{Y} + A Y - \frac{B}{2}Y^2,
	\]
	and so $\p_Y e = \frac{1}{Y^2} p(Y) = 0$ is equivalent to $p(Y) = 0$. Any ($x$-independent) minimizer of $E$ must be one of the $u_Y$ or $U_\infty$, as $E[u]$ is in the case $u\ne U_\infty$ lowered by replacing $u$ by $u_Y$ with $Y$ the smallest number for which $u(Y) = 0$. So $u_Y$ is a local minimizer of $E$ if and only if $Y$ is a local minimum of $e(Y)$.

	The roots of $p$ can be written in a closed-form expression, but it will be more useful for us to write
	\[
		p(Y) = 0 \qquad \Leftrightarrow \qquad Y^2 (A - B Y) = 1.
	\]
	This always has one negative solution, which is not relevant here, and $0, 1$, or $2$ solutions between $0$ and $A/B$. The maximum of $Y^2 (A - B Y)$ over $[0, A/B]$ is $ \frac{4}{27} \frac{A^3}{B^2}$, attained at $Y = \frac{2}{3} \frac{A}{B}$, leading to the following characterization:
	\begin{itemize}
		\item If $B < 2 (\frac{A}{3})^{3/2}$, there are two values $0 < Y_- < \frac{2}{3}\frac{A}{B} < Y_+ < \frac{A}{B}$ for which $p(Y) = 0$, with $p$ positive between them and negative for other positive $Y$. This means that for both $Y = Y_-, Y_+$, $u_Y$ is an inner variation critical point. 
		\item If $B =  2 (\frac{A}{3})^{3/2}$, there is exactly one value $Y = \frac{2}{3}\frac{A}{B}$ for which $p(Y) = 0$, still corresponding to an inner variation critical point for the same reason.
		\item If $B >  2 (\frac{A}{3})^{3/2}$, there are no positive roots of $p$ and no domain variation critical points of this type.
	\end{itemize}

	In the first case, as $\p_Y e$ is a positive multiple of $p$, it is easy to see that $Y_+$ is a local maximum and $Y_-$ is a local minimum of $e$. Whether or not $E[u_{Y_-}] < E[u_\infty]$ depends on the parameters $A, B$ as well, but we will not find it necessary to explicitly classify this.
	
	Let us verify that $U_-$ and $U_\infty$ are local minimizers to $E$: more precisely, we will show that if for a $u(y) \in H$
	\[
		\int |u_y - (U_-)_y|^2 < \d = \d(A, B),
	\]
	then $E[u] \geq E[U_-]$ (and then similarly for $U_\infty$).

	Let $Y = \inf \{y : u(y) = 0\}$. We claim that $Y > Y_- - C(A, B) \sqrt{\d}$: indeed,
	\[
		\frac{1}{Y_-} (Y_- - Y)  \leq |u(Y) - U_-(Y)| \leq \sqrt{\int_0^Y | (u - U_-)_y|^2} \sqrt{Y} < \sqrt{\d Y},
	\]
	so if $Y < Y_-$ we get that $Y \geq Y_- - C(Y_-)\sqrt{\d}$. By a similar argument, we also have that
	\[
		u(Y_-) \leq C \sqrt{\d}.
	\]
	If $Y < Y_- + C_*(A, B)$, we can use $l = (1 - y/Y)_+$ as a competitor:
	\[
		E[u] \geq E[l] \geq E[U_-],
	\]
	with the second inequality coming from the local minimality of $Y_-$ for $e$ above. On the other hand, if $Y > Y_- + C_*$, we instead use as a competitor 
	\[
		v(y) = \begin{cases}
		\frac{(u(y) - u(Y_-))_+}{1 - u(Y_-)} & y \leq Y_- \\
		0 & y > Y_-.
	\end{cases}
	\]
	Then using that $|v_y| \leq (1 + C \sqrt{\d}) |u_y|$,
	\[
		E[u] \geq \int_0^{Y_-} \left(u_y^2 + (A - B y)\right) + \int_{Y_-}^{Y_- + C_*} (A - B y)_+ \geq (1 - C \sqrt{\d}) E[v] + c(A, B) C_* \geq (1 - C \sqrt{\d}) E[U_-] + c C_*.
	\]
	The final inequality used that $U_-$ is harmonic on $[0, Y_-]$ and $U_- = v$ outside of $(0, Y_-)$, so $E[U_-] \leq E[v]$. Provided $\d$ is small relative to the other constants, this gives $E[u] > E[U_-]$.

	For the local minimality of $U_\infty$, the situation is in fact simpler: an analogous argument gives that if $\d$ is small enough, $Y \gg A/B$ in this case. Then $E[u] \geq E[l] \geq E[U_\infty]$ concludes the argument, as indeed $E[l] \geq E[U_\infty]$ for any $Y > A/B$. We omit the details.
\end{proof}

\begin{lemma}[Local minimality]\label{lem:localmin}

	\begin{enumerate}
		\item Let $u \in H$ be a minimizer of $E$, i.e.
		\[
			E[u] = \inf\{E[v] : v \in H\}.
		\]
		Then $u$ is independent of $x$, so in particular it is either $U_-$ or $U_\infty$ from Lemma \ref{lem:1d}.
		\item Assume that $B < 2 (\frac{A}{3})^{3/2}$. There is a $\d_0 = \d_0(A, B)$ such that if
		\[
			\|\nabla u - \nabla U_-\|_{L^2}  \leq \d_0 \qquad (\|\nabla u - \nabla U_\infty\|_{L^2}  \leq \d_0),
		\]
		then
		\[
			E[U_-] \leq E[u] \qquad (E[U_\infty] \leq E[u]),
		\]
		with equality only if $u = U_-$ ($u = U_\infty$).
	\end{enumerate}
\end{lemma}

\begin{proof}
	We denote by $U$ whichever of $U_-$ or $U_+$ has a smaller value of $E$; by Lemma \ref{lem:1d} we have that
	\[
		E[U] \leq \int_0^\infty \left(|v_x|^2 + \chi_{\{v > 0\}}(A - B y)_+\right)
	\]
	for any $v \in \dot{W}^{1,2}([0, \infty))$ with $v(0) = 1$. Then
	\[
		E[u] = \int_{\T \times [0, \infty)} |u_x|^2 + \int_\T \int_{ [0, \infty)} \left(|u_y|^2 + \chi_{\{u > 0\}}(A - B y)_+\right) \geq \int_{ \T \times [0, \infty)} |u_x|^2 + E[U]
	\]
	by applying this on almost every ray $[0, \infty) \times \{x\}$. Then if $u$ is a minimizer, we have $u_x = 0$ a.e., which implies $u$ is independent of $x$.

	For (2) we prove the local minimality of $U_-$, as the other is similar. First, the 1D version of this statement has already been shown in Lemma \ref{lem:1d}. By choosing $\delta_0$ small, we can ensure that if $A = \{x : \|\nabla u(x, \cdot) - \nabla U_-(\cdot)\|_{L^2} > \delta \}$,
	\[
		|A| < \frac{\|\nabla u - \nabla U_-\|_{L^2}^2}{\delta^2} \leq \frac{\delta_0^2}{\delta^2} < \delta.
	\]
	On $\T \sm A$, we have that $E[u(x, \cdot)] \geq E[U_-]$ as long as $\delta$ is small enough.

	Let $V_a = (1 - y/a)_+$. The computation in Lemma \ref{lem:1d} shows that $E[V_a] > E[U_-]$ on a neighborhood $a \in (0, Y_- + c_1(A, B))$.
	
	Take an $x$ with $E[u(x, \cdot)] < E[U_-]$, and set $Y(x) = \inf\{y : u(x, y) = 0\}$, for $\eta$ small. Then $Y > Y_- + c_1$, for $E[U_-] > E[u(x, \cdot)] \geq E[V_Y]$, which is only possible for $Y$ large enough. We claim something similar holds for $Z = \inf\{y : u(x, y) = \eta\}$, where $\eta \ll 1$ is small. Indeed, as $E[U_-] > E[u(x, \cdot)]$,
	\[
		\int_0^Z |u_y(x, y)|^2 dy < \int_0^\infty |(U_-)_y|^2 - \int_{Y_-}^{Y_- + c_1} (A - B y)_+ < \int_0^\infty |(U_-)_y|^2 - c_2(A, B) = \frac{1}{Y_-} - c_2.
	\]
	On the other hand, as $u(x, 0) = 1$ and $u(x, Z) = \eta$,
	\[
		\int_0^Z |u_y(x, y)|^2 \geq \frac{(1 - \eta)^2}{Z},
	\]
	so
	\[
		Z \geq Y_- \frac{(1 - \eta)^2}{1 - c_2 Y_-} \geq Y_- + c_3(A, B)
	\]
	provided $\eta$ is small enough relative to $A, B$.

	We have shown that for any $x$ with $E[u(x, \cdot)] < E[U_-]$, $u(x, y) \geq \eta$ for $y < Y_- + c_3$ (independent of $\d$), and also that there is another $x' \in \T \sm A$ with $|x - x'| < \d$. At $x'$, we must have that
	\[
		|u(x', y) - U_-(y)| \leq \int_0^y |(u(x', t) - U_-(t))_y| dt \leq \sqrt{\d y}.
	\]
	So on $\{x'\} \times [Y_-, Y_- + c_3]$ where $U_- = 0$, $u(x', y) \leq \frac{\eta}{2}$ as long as $\d$ is taken small enough. Now integrate in $x$ along an interval $I$ with endpoints $x, x'$ and $|I| < \d$:
	\[
		c(\eta, A, B) \leq \int_{Y_-}^{Y_- + c_3}|u(x', t) - u(x, t)| dt \leq \int_{I} \int_{Y_-}^{Y_- + c_3}|u_x(s, t)| dt \leq \sqrt{\d \int u_x^2} \leq \sqrt{\d E[U_-]},
	\]
	at the very end supposing for contradiction that $E[u] < E[U_-]$. For $\d$ taken small enough, this is a contradiction. It follows that, in fact, $E[u(x, \cdot)] \geq E[U_-]$ for \emph{all} $x$, and so after integrating we get that $E[u] \geq E[U_-]$, with equality only if $u_x = 0$ almost everywhere.
\end{proof}

\subsection{Second variation and Morse index}

We now compute the Morse index of the solution $U_+$, as well as second variation formulas in a limited context. The second variation for Bernoulli-type problems is well known and often used in regularity theory or shape optimization. In this case, we will be interested in only the second variation around the flat solution $U_+$ (which simplifies the computation) and the \emph{structure} of the second inner variation in general (to make sure it is stable under limits).

\begin{lemma}[Morse index close to $U_+$]\label{lem:variationformulas}

	\begin{enumerate}
		\item Let $u$ be an outer variation critical point for $E_\e$. Then for any $v \in W^{1, 2}_0(\T \times (0, T))$, the mapping $t \mapsto E_\e[u + t v]$ is smooth near $0$ and
		\[
			\partial_{tt} E_\e[u + t v] |_{\{t = 0\}}  = \int \left(2|\nabla v|^2 + 2\b_\e^{'}(u) v^2\right).
		\]
		\item Assume that $B < 2 (\frac{A}{3})^{3/2}$. For a smooth function $g \in C^\infty(\T)$, define the vector field $V = (V^x, V^y) : \T \times [0, Y_+] \rightarrow \R^2$ via
		\[
			\begin{cases}
				V^x = 0 &\\
				V^y(x, 0) = 0 \\
				V^y(x, Y_+) = g \\
				\Delta V^y = 0 & \text{ on } \T \times (0, Y_+),
			\end{cases}
		\]
		and then extend $V$ smoothly to $\T \times [0, \infty)$ so that it has compact support in $\T \times [0, A/B)$.	Let $\phi_t$ be the flow of $V$. Then for small $t$, $\phi_t$ is a diffeomorphism, $t \mapsto E[U_+ \circ \phi_t^{-1}]$ is smooth, and
		\[
			\begin{cases}
				\partial_{t}  E[U_+ \circ \phi_t^{-1}] |_{\{t = 0\}} = 0,\\
				\partial_{tt}  E[U_+ \circ \phi_t^{-1}] |_{\{t = 0\}} = \int_{\{ y < Y_+\} } \frac{2}{Y^2_+} |\nabla V^y|^2   - B \int_{\T} g^2(x) \> dx.
			\end{cases}
		\]
		\item If $\e_k \searrow 0$ and $u_{k}$ are critical points of $E_{\e_k}$ converging in $H$ topology to $U_+$, and with $\chi_{\{ u_k > 0\} } \rightarrow \chi_{\{ U_+ > 0\} }$ in $L^1$, then
		\[
			\partial_{tt}  E_{\e_k}[u_k\circ \phi_t^{-1}] |_{\{t = 0\}} \rightarrow \partial_{tt}  E[U_+ \circ \phi_t^{-1}] |_{\{t = 0\}}
		\]
		for any smooth $g \in C^\infty(\T)$ as above.
		\item If 
		\begin{align*}
			2 \left(\frac{A}{B} - Y_+ \right) 2 \pi \coth(2 \pi Y_+) < 1\\
			(\text{equivalently  } 2 \frac{A}{B} 2 \pi \left( 1 - \frac{1}{3}\left(1+2 \cos\left( \frac{1}{3} \arccos\left( 1-\frac{27B^{2}}{2A^{3}}\right) \right)\right)\right)\times\\
			\coth\left(2 \pi \frac{1}{3} \frac{A}{B} \left(1+2 \cos\left( \frac{1}{3} \arccos\left( 1-\frac{27B^{2}}{2A^{3}}\right) \right)\right)\right) < 1)
		\end{align*}
		and $u_{k}$ are as above, then for $k$ large $u_k$ has Morse index at least $2$: i.e. there is a two-dimensional subspace $W$ of $\dot{W}^{1, 2}_0$ such that for any $v \in W\sm \{0\}$, $\partial_{tt}  E_{\e_k}[u_k + t v] |_{\{t = 0\}} < 0$.
	\end{enumerate}
\end{lemma}
\begin{proof}
	Part (1) is the standard second outer variation for semilinear equations obtained similarly to the Euler-Lagrange equations.

	For part (2), we first observe that $|V(x, y)| \leq y \frac{\max |g|}{Y_+}$ by the maximum principle. From this it may be verified that for $t$ small enough $\phi_t$ maps $\T \times [0, \infty)$ into itself, and from this that it is bijective. 
	
	We compute an expansion for $E[Z \circ \phi_t^{-1}]$ for any smooth flow $\phi_t$ which is a diffeomorphism of $\T \times [0, A/B]$ to itself and any $Z \in H$. Set $\psi_t = \phi^{-1}_t$ and $Z_t = Z \circ \psi_t$. Then (subscripts are derivatives, repeated indices are summed over, $t$ subscripts are omitted):
	\[
		\begin{cases}
			\phi^j = x^j + t V^j + \frac{t^2}{2} V^j_k V^k + o(t^2) \\
			\phi^j_i = \d^j_i + t V^j_i + \frac{t^2}{2} (V^j_{ki} V^k + V^j_k V^k_i) + o(t^2) \\
			\psi^j_i = \d^j_i - t V^j_i + \frac{t^2}{2} (-V^j_{ki} V^k + V^j_k V^k_i) + o(t^2) \\
			\det D\phi = 1 + t \dvg V + \frac{t^2}{2} [V_{ki}^i V^k + (\dvg V)^2 ] = 1 + t \dvg V + \frac{t^2}{2} \dvg (V \dvg V).
		\end{cases}
	\]
	The Dirichlet energy can then be approximated by changing variables:
	\begin{align*}
		\int |\nabla Z_t|^2 &= \int Z_i Z_j \psi^i_k \psi^j_k |\det D \phi| \\
		&= \int |\nabla Z|^2 \\
		&+ t \int \left(- 2 V^i_j Z_i Z_j + |\nabla Z|^2 \dvg V \right)\\
		&+ \frac{t^2}{2} \int Z_i Z_j \left( - 4  V^j_i \dvg V + 2 V^i_k V^j_k + 2 (- V^j_{i k} V^k + V^j_k V^k_i)  + \d_i^j \dvg (V \dvg V)\right).
	\end{align*}
	The volume term can also be computed by changing variables (note that by our assumption $\phi$ is the identity for $y > A/B$):
	\begin{align*}
		\int_{\{ Z_t > 0\} } (A - B y)_+ &= \int_{\{ Z > 0\} } (A - B \phi^y)_+ |\det D \phi| \\
		&= \int_{\{ Z > 0\} } (A - B \phi^y)_+ \\
		&+ t \int_{\{ Z > 0\} } \left(- B V^y + (A - B y) \dvg V \right)\\
		&+ \frac{t^2}{2} \int_{\{ Z > 0\} } \left(- 2 B V^y \dvg V - B V^y_i V^i + (A - B y) \dvg (V \dvg V)\right).
	\end{align*}
	This leads to the expressions for derivatives of $E$:
	\[
	\begin{cases}
		\partial_{t} E[Z_t] |_{\{t = 0\}} = \int \left(- 2 V^i_j Z_i Z_j + |\nabla Z|^2 \dvg V  + \chi_{\{ Z > 0\} } \dvg \left( (A - B y) V\right) \right)\\
		\partial_{tt}  E[Z_t] |_{\{t = 0\}} = \int \biggl(Z_i Z_j \left( - 4  V^j_i \dvg V + 2 V^i_k V^j_k + 2 (- V^j_{i k} V^k + V^j_k V^k_i)  + \d_i^j \dvg (V \dvg V) \right)\\
		+ \chi_{\{ Z > 0\} } (- 2 B V^y \dvg V - B V^y_i V^i + (A - B y) \dvg (V \dvg V))\biggr).
	\end{cases}	
	\]
	A similar computation can be performed for $E_\e$; we omit the expressions, but let us observe if $u_\e \rightarrow Z$ in $H$ topology and also $\chi_{\{ u_\e > \e\} } \rightarrow \chi_{Z}$ in $L^1$, then the first and second variations of $E_\e$ at $u_\e$ converge to the first and second variations of $E$ for $Z$, for any vector field $V$ as described here. In particular this proves part (3).
	
	From these expressions it is also clear that both quantities are continuous under $C^2$ convergence of $V$, so if the $V$ of part (2) is approximated by ones compactly supported on $\T \times (0, A/B)$, for which the inner variation critical point property of $U_+$ gives that the first variation is $0$, we will recover that
	\[
		\partial_{t} E[U_+ \circ \phi_t^{-1}]  |_{\{t = 0\}} = 0.
	\]
	We now restrict our attention to that specific $V$, and simplify the formulas using (1) the fact that $V$ is harmonic on $\{U_+ > 0\}$ and (2) the explicit formula $U_+ = (1 - y/Y_+)_+$. Set $h(x) = V^y_y(x, Y_+)$ to be the normal derivative of $V^y$. Firstly,
	\[
		\int |\nabla U_+|^2 \dvg (V \dvg V) = \int_{\T} |\nabla U_+(x, Y_+)|^2 g h \> dx = \frac{1}{Y_+^2} \int_{\T} g h \> dx,
	\]
	using that $V^y(x, 0) = 0$. The other terms have simplified expressions due to $V^x = 0$ and $(U_+)_x = 0$:
	\begin{align*}
		\int (U_+)_i (U_+)_j &\left( - 4  V^j_i \dvg V + 2 V^i_k V^j_k + 2 (- V^j_{i k} V^k + V^j_k V^k_i)\right)\\ &= \frac{1}{Y^2_+} \int_{\{ 0<y< Y_+\} } \left(- 4 (V^y_y)^2 + 2 |\nabla V^y|^2 + 2 (- V^y_{yy} V^y  + (V^y_y)^2) \right)\\
		&= \frac{1}{Y^2_+} \left(\int_{\{ 0<y < Y_+\} }  2 |\nabla V^y|^2 - \int_{\T} 2 g h\> dx\right)
	\end{align*}
	after integrating by parts. The volume terms admit similar simplifications:
	\[
		\int_{\{ 0<y< Y_+\} } \left(- 2 B V^y \dvg V - B V^y_i V^i + (A - B y) \dvg (V \dvg V) \right)= \int_{\{ 0<y< Y_+\} } - 2 B V^y V^y_y + \int_\T g h (A - B Y_+)\> dx
	\]
	after integrating the rightmost term by parts. Then
	\[
		\int_{\{ 0<y< Y_+\} } - 2 B V^y V^y_y = \int_{\{ 0<y< Y_+\} } - B \left((V^y)^2\right)_y = \int_\T - B g^2 \> dx.
	\]
	Putting these together,
	\[
		\partial_{tt}  E[U_+ \circ \phi_t^{-1}] |_{\{t = 0\}} =  \int_{\{ 0<y< Y_+\} } \frac{2}{Y^2_+} |\nabla V^y|^2  + \int_{\T} [(A - B Y_+) - \frac{1}{Y^2_+}] g h - B g^2 \> dx.
	\]
	Recalling the free boundary condition from Lemma \ref{lem:FB}, $\frac{1}{Y_+^2} = A - B Y_+$, so in fact 
	\[
		\partial_{tt} E[U_+ \circ \phi_t^{-1}] |_{\{t = 0\}} =  \int_{\{ 0<y< Y_+\} } \frac{2}{Y^2_+} |\nabla V^y|^2   - B \int_{\T} g^2 \> dx.
	\]

	For part (4), we consider only $g$ which are even in $x$: $g(x) = g(-x)$. Let us use the notation $V^g$ for the vector field defined as above associated with $g$, and the bilinear form
	\[
		B[g_1, g_2] = \int_{\{ 0<y< Y_+\} } \frac{2}{Y^2_+} \nabla (V^{g_1})^y \cdot \nabla (V^{g_2})^y   - B \int_{\T} g_1 g_2 \> dx
	\]
	defined on even functions in $C^\infty(\T) \ss W^{1/2, 2}(\T)$. It is straightforward to explicitly diagonalize $B$ using Fourier series. For integers $m \geq 1$,
	\[
		g = \sqrt{2} \cos(2 \pi m x) \qquad g' = \sqrt{2}\cos(2 \pi m' x)
	\]
	the corresponding $V^g$ can be found by separating variables:
	\[
		V^{g} = (0, \frac{\sinh(2 \pi m y)}{\sinh(2 \pi m Y_+)} \sqrt{2} \cos(2 \pi m x) ).
	\]
	Then
	\[
		\int_{\{ 0<y< Y_+\} } \nabla (V^{g})^y \cdot \nabla (V^{g'})^y =  \int_\T g' V^{g}_y(x, Y_+) \> dx= \int_\T g g' (2 \pi m) \coth(2 \pi m Y_+) \> dx,
	\]
	so
	\[
		B[g, g'] = \d_{m}^{m'} [\frac{4 \pi m}{Y_+^2}\coth(2 \pi m Y_+) - B].
	\]
	For $m = 0$, we instead would have $g = 1$, $V^g = (0, y/Y_+)$, and $B[g, g'] = 0$ for $m' \geq 1$,
	\[
		B[1, 1] = \frac{2}{Y_+^3} - B = \frac{2}{Y_+} A - 3 B,
	\]
	using the free boundary condition $\frac{1}{Y_+^2} = A - B Y_+$. As $Y_+ > \frac{2}{3} \frac{A}{B}$, we have that $B[1, 1] < 0$.
	
	Therefore, in this Fourier basis $B$ is diagonal with eigenvalues $\frac{2}{Y_+^3} - B, \frac{2}{Y_+^3} (2 \pi m Y_+) \coth(2 \pi m Y_+) - B$. The function $t \coth(t)$ is increasing and converges to $1$ at $0$, so this is an increasing sequence. The second smallest eigenvalue is
	\[
		\frac{2}{Y_+^3} (2 \pi Y_+) \coth(2 \pi Y_+) - B = 2 (A - B Y_+ ) 2 \pi \coth(2 \pi Y_+) - B.
	\]
	If this is negative, then there is a two-dimensional space of vector fields $W$ generated by the $V^g$ for these first two eigenfunctions of $B$ such that for every $V \in W$,
	\[
		\partial_{tt}  E_{\e_k}[u_k\circ \phi_t^{-1}] |_{\{t = 0\}} < 0
	\]
	for all $\e_k$ sufficiently small. Using that $u_k$ is a critical point for $E_{\e_k}$ and everything is smooth, we have that
	\[
		\partial_{tt}  E_{\e_k}[u_k\circ \phi_t^{-1}] |_{\{t = 0\}} = \partial_{tt} E_{\e_k}[u_k + t \nabla u_k \cdot V]  |_{\{t = 0\}},
	\]
	giving a two-dimensional subspace of $W^{1,2}_0(\T \times [0, A/B])$ for the outer variation Morse index property.
	
Last, let us compute the largest of the three real roots of the cubic equation
\begin{equation*}
A Y^{2}-B Y^{3}-1=0.
\end{equation*}
Setting
$a:=\frac{A}{-B}, c:=\frac{1}{B}$,
the equation takes the standard form
\begin{equation}\label{eqstandard}
Y^{3}+aY^{2}+c=0.
\end{equation}
In order to remove the 
quadratic term (Tschirnhaus shift), 
put
\begin{equation*}
Y = z-\frac{a}{3},
\end{equation*}
which transforms (\ref{eqstandard}) into the \emph{depressed} cubic
\begin{equation*}
z^{3}+Pz+Q=0\qquad\text{with}\qquad P=-\frac{a^{2}}{3}, Q=\frac{2a^{3}}{27}+c.
\end{equation*}
In terms of the original parameters
\begin{equation*}
P=-\frac{A^{2}}{3B^{2}},\qquad Q=\frac{1}{B}-\frac{2A^{3}}{27B^{3}}.
\end{equation*}
The trigonometric Vi\`ete solutions are then
\begin{equation*}
Z_k = 2 \sqrt{-\frac{P}{3}} \cos\left( \frac{1}{3} \arccos\left( \frac{3Q}{2P} \sqrt{-\frac{3}{P}}\right) - k\frac{2\pi}{3}\right), k=0,1,2.
\end{equation*}
Because the cosine is decreasing on $[0,\pi]$, the \emph{largest} root is obtained for $k=0$:
\begin{equation*}
Z_0 = \frac{2}{3} \frac{A}{B} \cos\left( \frac{1}{3} \arccos\left( 1-\frac{27B^{2}}{2A^{3}}\right) \right).
\end{equation*}
Undoing the Tschirnhaus shift we obtain
\begin{equation}\label{eqY+}
Y_+ = \frac{1}{3} \frac{A}{B} \left(1+2 \cos\left( \frac{1}{3} \arccos\left( 1-\frac{27B^{2}}{2A^{3}}\right) \right)\right).
\end{equation}
\end{proof}
\begin{figure}[htbp]
    \centering
    \includegraphics[width=0.5\textwidth]{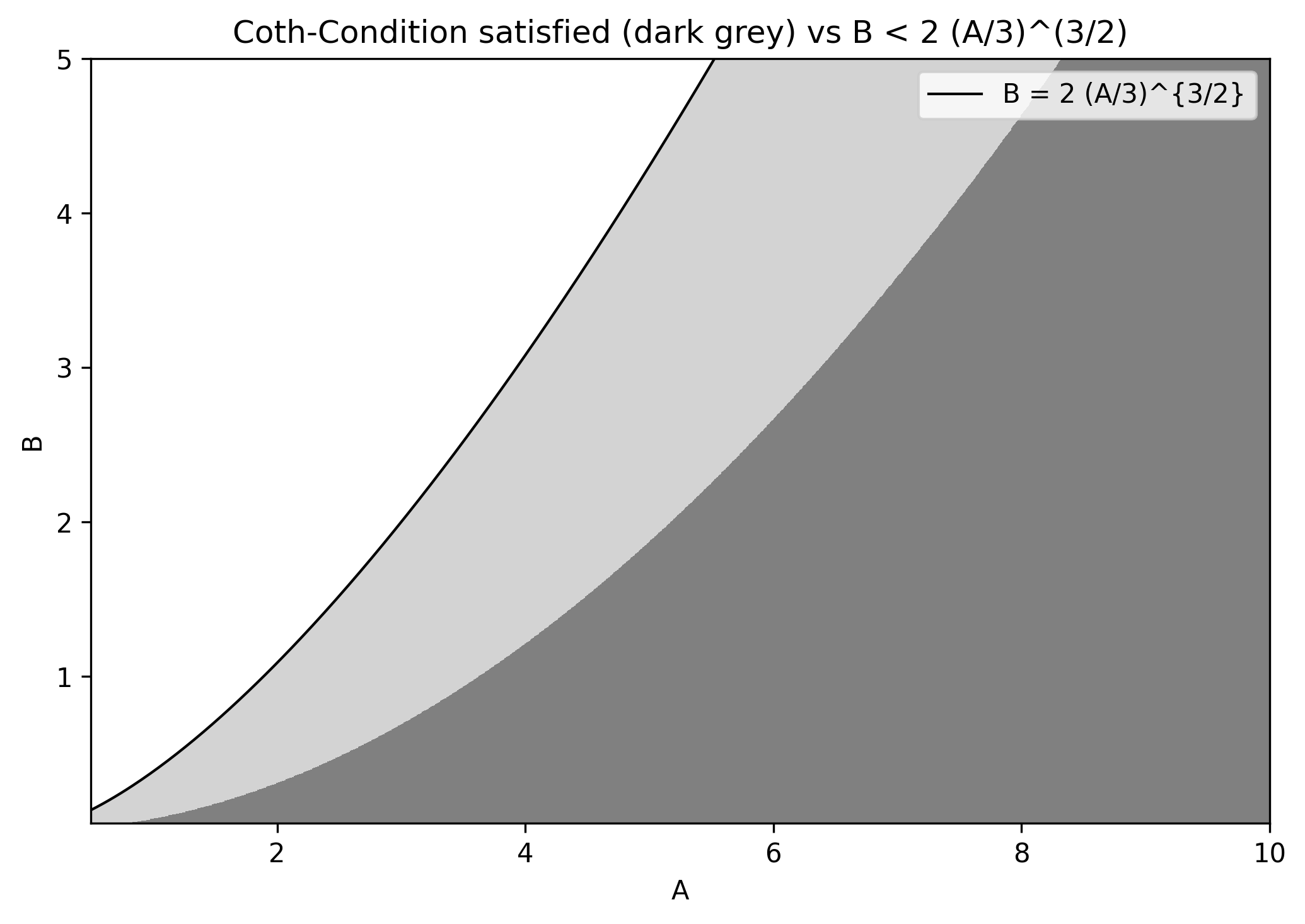} 
    \caption{The admissible region in the $AB$-plane}
    \label{fig:abplot}
\end{figure}
\begin{remark}\label{aboutcondition2}
To better understand the condition imposed in part (4), consider the ratio $A/B$ fixed. The condition $B < 2 (\frac{A}{3})^{3/2}$ can be rewritten
\[
	A > \left(\frac{B}{A}\right)^2 \frac{27}{4}.
\]
The root $Y_+$ moves from $\frac{2}{3}\frac{A}{B}$ to $\frac{A}{B}$ as $A$ increases from $(\frac{B}{A})^2 \frac{27}{4}$ to infinity. The function
\[
	g(t) :=  2 \frac{A}{B} (1 - t) 2 \pi \coth\left(2 \pi t \frac{A}{B}\right)
\]
is decreasing in $t$ and is $0$ at $t = 1$ (corresponding to the limit of large $A$).
So the condition in part (4), which reads $g(Y_+ B/A) < 1$, is satisfied when $A \geq A_* = A_*(A/B) > (\frac{B}{A})^2 \frac{27}{4}$.
In other words, the condition is always satisfies when the wave speed $> \frac{B}{A} \frac{3 \sqrt{3}}{2}$.
\end{remark}

\subsection{The relaxed functional}

The energy landscape for $E_\e$ is potentially more complicated, especially for large $\e$, and is not particularly relevant to our analysis. Instead, we exploit that $E_\e \rightarrow E$ to pass some information to $E_\e$.

\begin{lemma}[Gamma convergence]\label{lem:gamma}
	$E_\e$ $\gamma$-converges to $E$ with respect to the weak topology on $H \ss \dot{W}^{1,2}$. More precisely:
	\begin{enumerate}
		\item If $u_k \rightarrow u \in H$ weakly in $H$ and $\e_k \searrow 0$, then
		\[
			E[u] \leq \liminf_k E_{\e_k}[u_k].
		\]
		\item For any $u \in H$ and $\e_k \searrow 0$, there exist $u_k \in H$ with $u_k \rightarrow u$ weakly in $H$ such that
		\[
			E[u] \geq \limsup_k E_{\e_k}[u_k].
		\]
	\end{enumerate}
\end{lemma}

\begin{proof}
	The limsup inequality is in fact trivial: we have that for each $u \in H$, $E_\e[u]$ is non-increasing in $\e$ and converges to $E[u]$, so setting $u_k = u$ gives $E[u] \geq E_{\e_k}[u_k]$.

	For the liminf inequality, first observe that from compact embeddings we may assume that $u_k \rightarrow u$ strongly in $L^2_{\text{loc}}$ and almost everywhere. For each sequence of numbers $t_k \rightarrow t$, we have that $\chi_{(0, \infty)}(t) \leq \liminf_k 2 \cB_{\e_k}(t_k)$. Indeed, the inequality is trivial if $t \leq 0$, while if $t > 0$ then for $k$ large enough $t_k > t/2$ and so for $k$ even larger $\cB_{\e_k}(t_k) \geq \cB_{\e_k}(t/2) = 1$. Applying Fatou's lemma,
	\[
		\int \chi_{\{ u > 0\} } (A - B y)_+ \leq \liminf_k \int 2\cB_{\e_k}(u_k) (A - B y)_+.
	\]
	The other term in $E_\e$ is lower semicontinuous under weak converge of $\nabla u_k$, so we obtain
	\[
		E[u] \leq \liminf_k E_{\e_k}[u_k].
	\]
\end{proof}

A standard consequence of $\gamma$-convergence and the local minimality of $U_-$ and $U_\infty$ is the following stability lemma.

\begin{lemma}[Stability property]\label{lem:stable}
	Assume that $B < 2 (\frac{A}{3})^{3/2}$. Then for every $\d < \d_0$, there exists an $\eta = \eta(\d, A, B) > 0$ such that if
	\[
		\|\nabla u - \nabla U_-\|_{L^2} \in [\d, \d_0],
	\]
	then
	\[
		E_\e[u] \geq E_\e[U_-] + \eta.
	\]
	for all $\e < \e_1 = \e_1(\d, A, B)$. The same is true with $U_\infty$ in place of $U_-$.
\end{lemma}

Note that we do not claim that $U_-, U_\infty$ are local minimizers of $E_\e$; indeed, this is clearly false for $U_-$.

\begin{proof}
	We argue by contradiction: if not, then there is a $\d > 0$ and sequences $u_k \in H$, $\e_k \searrow 0$ as $k\to\infty$, such that
	\[
		\|\nabla u_k - \nabla U_-\|_{L^2} \in [\d, \d_0]
	\]
	but
	\[
		E_{\e_k}[u_k] < E_{\e_k}[U_-] + \frac{1}{k} \text{ for all }k.
	\]
	We extract a subsequence $u_k \rightarrow u \in H, k\to\infty$ in the weak topology of $H$. By Lemma \ref{lem:gamma},
	\[
		E[u] \leq \liminf_{k\to\infty} E_{\e_k}[u_k] \leq \liminf_{k\to\infty} E_{\e_k}[U_-] = E[U_-].
	\]
	We also have that
	\[
		\|\nabla u - \nabla U_-\|_{L^2} \leq \liminf_{k\to\infty} \|\nabla u_k - \nabla U_-\|_{L^2} \leq \d_0.
	\]
	From Lemma \ref{lem:localmin}, then,
	\[
		E[u] = E[U_-],
	\]
	and so $u = U_-$. Moreover, as each term in $E_\e$ is separately lower semicontinuous, this implies that
	\[
		\int |\nabla u|^2 = \lim_{k\to\infty} \int |\nabla u_k|^2 ,\qquad \int \chi_{\{ u > 0\} } (A - B y)_+ =  \lim_{k\to\infty} \int 2\cB_{\e_k}(u_k) (A - B y)_+,
	\]
	so $\nabla u_k \rightarrow \nabla u = \nabla U_-$ strongly as $k\to\infty$. This is a contradiction to
	\[
	\|\nabla u_k - \nabla U_-\|_{L^2} \geq \d.
	\]
\end{proof}

\section{Mountain pass solutions}\label{s:mountainpass}

In this section, we produce critical points of $E_\e$. This amounts to applying standard mountain pass results, for which the key assumption needed is the Palais-Smale condition below.

\begin{lemma}[Palais-Smale condition] \label{lem:PS}
	$E_\e$ satisfies the Palais-Smale condition: if $u_k \in H$ is a sequence with $\sup_k E_\e[u_k] < \infty$ and
	\[
		\sup \{ \left|\int \left(\nabla u_k \cdot \nabla v + \b_\e(u_k) (A - B y)_+ v\right)\right| : v(x, 0) = 0, \|v\|_{\dot{W^{1, 2}}} \leq 1\} \rightarrow 0 \text{ as }k\to\infty,
	\]
	then $u_k$ has a subsequence converging strongly in $H$.
\end{lemma}

\begin{proof}
	As $E_\e [u_k]$ is bounded, we may find a subsequence which converges to $u \in H$ weakly in $\dot{W}^{1,2}$ topology, as well as in $L^2_{\text{loc}}$ and almost everywhere. Using $v = \frac{u_k - u}{\|\nabla u_k - \nabla u\|_{L^2}}$ as a test function for the second assumption,
	\[
		\left|\int \left(\nabla u_k \cdot \nabla (u_k - u) + \b_\e(u_k)(A - B y)_+ (u_k - u)\right)\right| \to 0 \text{ as }k\to\infty.
	\]
	Using the weak convergence,
	\[
		\int \nabla u \cdot \nabla (u_k - u) \rightarrow 0 \text{ as }k\to\infty.
	\]
	The function $\b_\e(u_k)(A - B y)_+ (u_k - u)$ is supported on $\{y < A/B\}$, and is uniformly integrable (it is uniformly bounded in $L^2$, recalling that $\b_\e$ is bounded), so its integral goes to $0$. We obtain that
	\[
		\int |\nabla (u_k - u)|^2 \rightarrow 0 \text{ as }k\to\infty,
	\]
	so $u_k \rightarrow u$ strongly as $k\to\infty$.
\end{proof}

Let $P = \{p \in C([0, 1]; H) : p(0) = U_-, p(1) = U_\infty\}$ and
\[
	G_\e = \inf_{p \in P} \sup_{t \in [0, 1]} E_\e [p(t)] \geq \max \{E_\e [U_-], E_\e [U_\infty]\}.
\]

\begin{lemma}[Mountain pass] \label{lem:mountainpass}
	Assume that $B < 2 (\frac{A}{3})^{3/2}$ and $\e < \e_2 = \e_2(A, B)$. Then there exists an outer variation critical point $u$ of $E_\e$ with $E_\e[u] = G_\e$, and
	\[
		G_\e > \max \{E_\e [U_-], E_\e [U_\infty]\} + \eta_2(A, B).
	\]
	Moreover, $u$ has Morse index at most $1$: given any two linearly independent functions $v_1, v_2 \in W^{1, 2}_0(\T \times [0, T])$, there is a linear combination $v = a_1 v_1 + a_2 v_2$ of them
	such that
	\[
		\p_{tt} E_\e[u + tv] |_{\{ t = 0\} } \geq 0.
	\]
\end{lemma}

\begin{proof}
	We apply the mountain pass theorem (see \cite[Theorem (2.)6.1]{S08}  or \cite{AR73}). To be precise, if $E_\e [U_-] \geq E_\e [U_\infty]$, set $V = \{ u \in \dot{W}^{1,2}(\T \times [0, \infty) : u(x, 0) = 0\}$ and $E_*[v] = E_\e[U_- + v] - E[U_-] : V \rightarrow \R$. By Lemma \ref{lem:PS}, $E_*$ satisfies the Palais-Smale condition. Fix a $\d < \d_0$ so that $\|U_- - U_\infty\|_V > \d$ and apply Lemma \ref{lem:stable} to get that, so long as $\e < \e_1$,	\[
		\|v\|_V = \d \qquad \implies \qquad E_*[v] \geq \eta > 0.
	\]
	We also have that
	\[
		E_*[U_\infty - U_-] \leq 0 < \eta.
	\]
	Then the mountain pass theorem directly applies to give a critical point $v_*$ of $E_*$ with $E_*[v_*] = G_\e - E[U_-] \geq \eta$; then $u_* = U_- + v_*$ is a critical point of $E_\e$ as desired. If $E_\e [U_-] < E_\e [U_\infty]$, swap the roles of $U_-$ and $U_\infty$.

	We note that $v \mapsto E_*[v]$ is $C^{k}$ for any $k$, as can be verified directly from the definition using that $\cB_\e$ is smooth. Then the result of \cite{FG92} shows it is possible to take $u_*$ to be of Morse index at most $1$.
\end{proof}

\section{Lipschitz bounds}\label{s:lip}

To pass to the limit in $\e$, some uniform estimates are needed. Bernoulli free boundary problems in general admit an a priori Bernstein-type Lipschitz estimate for critical points. Heuristically, the idea is that $|\nabla u|^2$ is subharmonic, controlled on the free boundary by the free boundary condition itself, and controlled on $\{y = 0\}$ by elementary barrier arguments (see \cite{ac}, for example). In practice we need this for the semilinear approximating problems, so sketch the argument below.

\begin{lemma}[Uniform Lipschitz estimate]\label{lem:lip}
	Fix $\e \in  (0,\e_3(A, B))$. There is a constant $M = M(A, B, \beta)$ (independent of $\e$) such that if $u_\e$ is an outer variation critical point of $E_\e$, then
	\[
		\sup_{\T \times [0, \infty)} |\nabla u_\e| \leq M.
	\]
\end{lemma}

\begin{proof}
We will omit the subscript of $u_\e$ in this proof, setting $u := u_\e$.

First, note that the maximum of $|\nabla u|$ cannot be attained at the boundary $y = 0$.
Indeed, since $u\in C^2( \T \times [0, \infty))$, $u$ is harmonic in an open neighborhood of
$y = 0$ relative to $\T \times [0, \infty)$, and
$$ \partial_y |\nabla u|^2 = 2 u_x u_{xy} + 2 u_y u_{yy} = 0 - 2 u_y u_{xx} = 0 \text{ on } y = 0.$$
However, supposing towards a contradiction that $m := \max_{\T \times [0, \infty)} |\nabla u|^2 = |\nabla u(x_0,0|^2$,
Hopf's principle would imply that the non-negative superharmonic function 
$m-|\nabla u|^2$ 
satisfies $\partial_y  (m-|\nabla u|^2) >0$ on $y = 0$, a contradiction.

Next, we check that $\max_{x \in \T} |\nabla u(x, y)|\rightarrow 0$ as $y \rightarrow \infty$. Indeed, for $y > A/B$ we have that $u$ is harmonic and bounded, and so can be represented by separation of variables:
	\[
		u(x, y) = a_0 + \sum_{k = 1}^\infty \left(a_k \sin(2 \pi k x) + b_k \cos(2 \pi k x)\right) e^{- 2 \pi k (y - A/B)}
	\]
	for $y > A/B$, with $a_0^2 + \sum_k a_k^2 + b_k^2 = \int_\T u^2(x, A/B) \leq 1$. It is straightforward to then verify that all derivatives decay exponentially.
	
Finally, we will prove a quantitative estimate of $|\nabla u|$ on $\{u \leq  \e \}$.
For the energy 
\[
\int  \left(|\nabla u|^2 + 2\cB_\e(u)\right)
\] 
this has been done in \cite{bcn},
and in a parabolic two-phase setting estimates for variable coefficients has been
proved in \cite{ck98}. For the sake of completeness we give a short proof here.

Let $(x_0,y_0)\in \{u \leq \e \}$ and let $v := u((x_0,y_0)+\e \cdot)/\e$. Then $|\Delta v|\leq C(A,B,\beta) (A-B y_0)_+
\leq C(A,B,\beta) A$ in $B_1(0)$, 
so that the Harnack inequality together with $C^{1,\alpha}$ estimates imply that $|\nabla u(x_0,y_0)|=|\nabla v(0)| \leq C(A,B,\beta)$.

Finally, we are in a position to conclude: Let $D = \{(x, y) : u(x, y) > \e, y < T\}$. As $|\nabla u|^2$ is subharmonic on $D$, and for large enough $T$ we have shown that $|\nabla u|^2 \leq C(A,B,\beta)$ on $\partial D$, the maximum principle gives that $|\nabla u|^2 \leq C(A,B,\beta)$ on $D$.
\end{proof}

In fact, a much stronger estimate is available, giving a sharp bound on the gradient. Estimates like this tend to hold for entire solutions, or asymptotically near free boundary points (like in \cite{ac}), but here the specific constant boundary condition means it is true globally.

\begin{lemma}[Strong Bernstein inequality]\label{lem:bern}
	Every weak $W^{1,2}$-limit $u$ of outer variation critical points $u_\e$ of $E_\e$ as $\e \to 0$
	satisfies 
	\begin{equation}
|\nabla u|^2 \leq (A-By)_+ \text{ a.e. on } \T \times [0, \infty).
\end{equation}
\end{lemma}
\begin{proof}
Suppose towards a contradiction that 
$$\limsup_{k\to \infty} \sup_{(x,y)\in \left(\T \times [0, \infty)\right)\cap  \{ \text{dist}(\cdot, \{ u_\e \leq \e\})>\epsilon\}}\left(|\nabla u_\e|^2 - (A-By)_+\right) = \delta > 0.$$
Since $|\nabla u_\e|^2 - (A-By)_+ = |\nabla u_\e|^2$ for $y\geq A/B$,
the above $\sup$ must by the proof of the previous lemma be attained in $\T \times [0,A/B]$.
Moreover, by the uniform Lipschitz continuity, $u_\e$ is harmonic and smooth on $\T \times [0, \kappa]$
for some positive $\kappa$ independent of $\e$. As in the proof of the previous lemma,
$|\nabla u_\e|^2$ attains its supremum on $\T \times [0, \kappa]$ at the boundary $y=\kappa$.
However, $|\nabla u_\e|^2 - (A-By)_+$ attaining a maximum at $y=0$
would imply that $|\nabla u_\e|^2$ attains its maximum at $y=0$, too, which
we have already excluded above.

Since $|\nabla u_\e|^2 - (A-By)_+$
is subharmonic in the interior of the set where $u_\e$ is harmonic,
\[
\sup_{(x,y)\in \left(\T \times [0, \infty)\right)\cap  \{ \text{dist}(\cdot, \{ u_\e \leq \e\})>\epsilon\}}\left(|\nabla u_\e|^2 - (A-By)_+\right)
\]
is attained on the set $\left(\T \times [\kappa,A/B]\right)\cap \{ \text{dist}(\cdot, \{ u_\e \leq \e\})=\epsilon\}$.
By the uniform Lipschitz continuity, we may take a sequence of points $(x_k,y_k) \in \left(\T \times [\kappa,A/B]\right)
\cap \{ \text{dist}(\cdot, \{ u_\e \leq \e\})=\epsilon\}$
such that
$$\limsup_{k\to \infty} \left(|\nabla u_{\e_k}(x_k,y_k)|^2 - (A-By_k)\right)
= \delta$$
and $$\sup_{B_{\e_k/2}(x_k,y_k)} \left(|\nabla u_{\e_k}|^2 - (A-By)\right) 
\leq o(1) + |\nabla u_{\e_k}(x_k,y_k)|^2 - (A-By_k)$$
as $k\to\infty$.
The sequence $v_k = u_{\e_k}((x_k,y_k)+ \e_k \cdot)/\e_k$ satisfies then
\begin{equation*}
	\begin{cases}
		v_k \geq 0,\\
		\Delta v_k =  \beta(v_k)(A-B(y_k+\e_k y)),\\
		\Delta v_k =  0 \text{ in }B_1(0),\\
		v_k > 1 \text{ in }B_1(0),\\ 
		v_k(w_k,z_k) = 1 \text{ for some  } (w_k,z_k) \in \partial B_1(0),\\
		\limsup_{k\to \infty} \left(|\nabla v_k(0)|^2 - (A-By_0)\right) \geq  \delta > 0,\\
		 \sup_{B_{1/2}(0)} \left(|\nabla v_k|^2 - (A-B(y_k+\e_k y))\right) 
	 \leq o(1) + |\nabla v_k(0)|^2 - (A-B y_k).
	\end{cases}
\end{equation*}
Passing to a limit $v_0, \xi_0, (x_0,y_0)$ for a subsequence
we obtain that
\begin{equation*}
	\begin{cases}
v_0 \geq 0,\\
   \Delta v_0 =  \beta(v_0)(A-B y_0),\\
    \Delta v_0 =  0 \text{ in }B_1(0),\\
   v_0 \geq 1 \text{ in }B_1(0),\\ 
   v_0(w_0,z_0) = 1 \text{ for some  }(w_0,z_0) \in \partial B_1(0),\\
   |\nabla v_0(0)|^2 - (A-By_0) \geq  \delta > 0,\\
    \sup_{B_{1/2}(0)} \left(|\nabla v_0|^2 - (A-By_0)\right) 
\leq |\nabla v_0(0)|^2 - (A-By_0).
\end{cases}
\end{equation*}
The strong maximum principle implies that the (in $\{v_0>1\}$) subharmonic
function $|\nabla v_0|^2$ is constant 
and that the (in $\{v_0>1\}$) harmonic
function is affine linear
in the connected component
of $\{v_0>1\}$ containing $0$.
There is a half-plane containing $0$ and touching
$(w_0,z_0)$ such that $v_0>1$ in that half plane
and $v_0=1$ on the boundary of that half plane.
By the unique continuation theorem \cite{aron}
(applied to the linear Schr\"odinger equation satisfied by the difference of any
two solutions as above),
$v_0$ is after rotation a function of one variable $x$
satisfying $v_0(0)=1$, $v_0''= \beta(v_0)(A-B y_0)$
and $v_0'(0) > A-B y_0\geq 0$.
Multiplying the ODE by $v_0'$, we obtain
$$(v_0')^2 -  2\cB(v_0)(A-B y_0) = const.$$
We distinguish two cases: if the solution $v_0$
has a critical point $x_0$, then we obtain
$$(v_0')^2 -  2\cB(v_0)(A-B y_0) = -  2\cB(v_0(x_0))(A-B y_0),$$
implying that $(v_0')^2 \nearrow (A-B y_0) (1-2\cB(v_0(x_0)))$
as $|x-x_0|\nearrow +\infty$, a contradiction to $ |\nabla v_0(0)|^2 - (A-By_0) \geq  \delta > 0$.
If the solution $v_0$ is increasing in $x$, then the non-negativity of $v_0$
implies that $v_0(x) \to 0$ as $x \to -\infty$, $(v_0')^2 -  2\cB(v_0)(A-B y_0) = 0$,
$(v_0')^2 \nearrow A-B y_0$ as $x  \nearrow +\infty$,
a contradiction to $ |\nabla v_0(0)|^2 - (A-By_0) \geq  \delta > 0$.
Thus the supposition at the beginning of the proof must be false.

We obtain that at each point at which the Lipschitz continuous limit function $u$ is
positive, $|\nabla u|^2 \leq (A-By)_+$. Since $|\nabla u|^2=0$ a.e. on the set
$\{ u=0 \}$, we obtain the statement of the theorem. 
\end{proof}
\begin{corollary}[No fluid above $A/B$]
Every weak $W^{1,2}$-limit $u$ of outer variation critical points $u_\e$ of $E_\e$ as $\e \to 0$
	satisfies $u=0$ in $y \geq A/B$.
\end{corollary}
\begin{proof}
By Lemma \ref{lem:bern}, $\nabla u=0$ in $y \geq A/B$.
But then $u$ is constant in $y \geq A/B$.
In the case that the constant is positive,
$u$ is constant
in the connected component 
of $\{ u>0 \}$ containing $y \geq A/B$.
In this case it follows that $u\equiv 1=U_\infty$.
But Lemma \ref{lem:mountainpass} tells us that $u\ne U_\infty$. So $u=0$ in $y \geq A/B$.
\end{proof}

\section{Existence of non-flat critical points}\label{s:existence}

In this section we prove Theorem \ref{main1} and the existence portion of Theorem \ref{main2}, starting with the former.

\begin{lemma}[Existence of non-flat limit]
	Assume that $B < 2 (\frac{A}{3})^{3/2}$ and also 
	\[
		2 (\frac{A}{B} - Y_+ ) 2 \pi \coth(2 \pi Y_+) < 1.
	\]
	Then there exists a domain variation critical point of $E$ other than $U_+, U_-, U_\infty$ (and in particular, it is not independent of $x$).
\end{lemma}

\begin{proof}
	Let $u_\e$ be the outer variation critical points from Lemma \ref{lem:mountainpass}. By Lemma \ref{lem:lip}, they satisfy $|\nabla u_\e|\leq C$, with $C$ independent of $\e$. We may then extract a sequence $u_{\e_k}$ converging locally uniformly to a $u \in H$
	as $k\to\infty$. It is easy to see that $u$ is harmonic when positive, which implies that the sequence converges strongly in $\dot{W}^{1,2}(\T \times [0, \infty))$: for any smooth nonnegative test function $\eta \in C_c^\infty(\T \times (0, \infty))$,
	\[
		\int |\nabla u_\e|^2 \eta = - \int \left(u_\e \eta \Delta u_\e + u_\e \nabla u_\e \cdot \eta \right)= - \int \left(u_\e \eta \b_\e(u_\e) + u_\e \nabla u_\e \cdot \nabla \eta\right).
	\]
	The first term is negative, so
	\[
		\limsup_{k\to\infty} \int |\nabla u_{\e_k}|^2 \eta \geq \limsup_{k\to\infty} \left(- \int u_{\e_k} \nabla u_{\e_k} \cdot \nabla \eta \right)= - \int u \nabla u \cdot \nabla \eta = \int |\nabla u|^2 \eta.
	\]
	This implies strong convergence of $\nabla u_\e \rightarrow \nabla u$ as $k\to\infty$.
	
	Now take any $U \cc \T \times [0, A/B]$ and apply Remark \ref{r:compactness}: we have that either $2\cB_{\e_k}(u_{\e_k}) \rightarrow \chi_{\{u > 0\}}$ in $L^1(U)$ as $k\to\infty$, or $u \equiv 0$ on $U$. By selecting $U$ of the form $\T \times (\d, A/B -\d)$ and (from the Lipschitz estimate) using that $u > 0$ near $y = \d$, we see that in fact $2\cB_{\e_k}(u_{\e_k}) \rightarrow \chi_{\{u > 0\}}$ a.e. on $[0, A/B]$ as $k\to\infty$. Then $2\cB_{\e_k}(u_{\e_k}) (A - B y)_+ \rightarrow \chi_{\{ u > 0\} } (A - B y)_+$ in $L^1(\T \times [0, \infty))$ as $k\to\infty$, and so we may pass to the limit in the domain variation formula to get that $u$ is a domain variation critical point of $E$.

	Passing in the limit in the energy, $E[u] \geq \max \{E [U_-], E [U_\infty]\} + \eta_2(A, B)$, so certainly $u$ is not $U_-$ or $U_\infty$. If $u = U_+$, then by Lemma \ref{lem:variationformulas} $u_{\e_k}$ has Morse index at least $2$ for $k$ large, which contradicts Lemma \ref{lem:lip}. From Lemma \ref{lem:1d}, it follows that $u$ is not independent of $x$.
\end{proof}

To prove Theorem \ref{main2}, we will need to apply Steiner rearrangements in the min-max procedure. We recall the setup now. Given a Borel set $I \ss \T$, let the \emph{symmetric rearrangement} $I^*$ be defined by
\[
	I^* = \begin{cases}
		\emptyset & |I| = 0 \\
		[-|I|/2, |I|/2] & |I| \neq 0.
	\end{cases}
\]
For a Borel subset $F \ss \T \times [0, \infty)$, let the \emph{Steiner rearrangement} $F^*$ be obtained by performing a symmetric rearrangement for every $y$:
\[
	F^* = \cup_{y \in [0, \infty)} \{x : (x, y) \in F\}^* \times \{y\}.
\]
Finally, for a $u \in W^{1,2}(\T \times [0, \infty))$, let the \emph{Steiner rearrangement} $u^*$ be given by
\[
	u^*(x, y) = \sup \{t \in \R : (x, y) \in \{u \geq t\}^*  \}.
\]
Under this definition $u^*$ is an upper semicontinuous function. It is easy to see from the definition and Fubini's theorem that
\[
	\int 2\cB_\e(u) (A - B y)_+ = \int 2\cB_\e(u^*) (A - By)_+ , \qquad \int \chi_{\{u > 0\}} (A - B y)_+ = \int \chi_{\{u^* > 0\}} (A - By)_+.
\]
On the other hand, the Polya-Szeg\"o inequality gives that the Dirichlet energy of $u^*$ is smaller than that of $u$.

\begin{proposition}\label{prop:sym} Let $u \in H$. Then $u^* \in H$, $E[u^*] \leq E[u]$, and $E_\e[u^*] \leq E_\e[u]$.
\end{proposition}

A proof can be found in \cite{K85} Theorem 2.31.

\begin{proposition}\label{prop:symcont}
	The map $u \mapsto u^*$ is a continuous function $H \rightarrow H$.
\end{proposition}

As $u \mapsto u^*$ is nonlinear, this is not a consequence of the Polya-Szeg\"o inequality and turns out to be extremely subtle (the corresponding result for radially decreasing rearrangement is false \cite{AL89}). The proof is found in \cite{B97}.

\begin{lemma}\label{lem:mountainpassimproved}
	The critical points in Lemma \ref{lem:mountainpass} can be taken to have $u_\e = u^*_\e$.
\end{lemma}

\begin{proof}
	Take $p_k \in P$ a sequence such that $\max_{t \in [0, 1]} p_k(t) \rightarrow G_\e$. Then let $p_k^*(t) := \left(p_k(t)\right)^*$ be the Steiner-symmetrized curve. By Proposition \ref{prop:symcont}, $p_k^* \in P$ (i.e. it is continuous in $t$), and by Proposition \ref{prop:sym}
	\[
		G_\e \leq \max_{t \in [0, 1]} p_k^*(t) \leq \max_{t \in [0, 1]} p_k(t) \rightarrow G_\e.
	\]
	Then the mountain pass theorem can be applied to this min-max sequence to obtain a critical point $u_\e$ with $\|u_\e - v_k\|_H \rightarrow 0$, where $v_k \in p_k([0,1])$ (see Theorem 1 in \cite{FG92}). The property of having $u_\e^* = u_\e$ is equivalent to being even, $u_\e(x, y) = u_\e(-x, y)$, and symmetrically decreasing, $(u_\e)_x(x, y) \leq 0$ for $x \in (0, 1/2)$. Both these properties are clearly preserved by convergence in $W^{1,2}_{\text{loc}}$.
\end{proof}

\begin{lemma}[Existence of non-flat symmetrically decreasing limit]\label{lem:symsol}
	Assume that $B < 2 (\frac{A}{3})^{3/2}$ and also $2 (\frac{A}{B} - Y_+ ) 2 \pi \coth(2 \pi Y_+) < 1$. Then there exists a domain variation critical point $u$ of $E$ other than $U_+, U_-, U_\infty$ (and in particular, it is not independent of $x$),
	and $u=u^*$.
	In particular, the critical point $u$ satisfies $u(x, y) = u(-x, y)$ and $u_x(x, y) \leq 0$ for a.e. $x \in (0, 1/2)$.
	The free boundary $\partial \{ u>0\}$ is the graph of a function of $y$, that is,
	$\partial \{ u>0\} = \{ (f(y),y): y \in S\}$, where $S$ is a closed subset of $[0,A/B]$.	 
\end{lemma}

\begin{proof}
	Let $u_\e$ be the outer variation critical points from Lemma \ref{lem:mountainpassimproved}. By Lemma \ref{lem:lip}, they satisfy $|\nabla u_\e|\leq C$, with $C$ independent of $\e$. We may then extract a sequence $u_{\e_k}$ converging locally uniformly to a $u \in H$ as $k\to\infty$. Moreover, $u = u^*$. It is easy to see that $u$ is harmonic when positive, which implies that $u_{\e_k} \rightarrow u$ strongly in $\dot{W}^{1,2}(\T \times [0, \infty))$ as $k\to\infty$: for any smooth nonnegative test function $\eta \in C_c^\infty(\T \times (0, \infty))$,
	\[
		\int |\nabla u_\e|^2 \eta = - \int \left(u_\e \eta \Delta u_\e + u_\e \nabla u_\e \cdot \eta \right)= - \int \left(u_\e \eta \b_\e(u_\e) + u_\e \nabla u_\e \cdot \nabla \eta\right).
	\]
	The first term is negative, so
	\[
		\limsup_{k\to\infty} \int |\nabla u_{\e_k}|^2 \eta \geq \limsup_{k\to\infty} \left(- \int u_{\e_k} \nabla u_{\e_k} \cdot \nabla \eta \right)= - \int u \nabla u \cdot \nabla \eta = \int |\nabla u|^2 \eta.
	\]
	This implies strong convergence of $\nabla u_{\e_k} \rightarrow \nabla u$ as $k\to\infty$.

\end{proof}

\section{Regularity of the water surface}\label{s:regularity}

In this section we prove the remaining portions of Theorem \ref{main2}, having to do with regularity of the free boundary.

\begin{definition}
Let us distinguish the water surface
$\S := \{ (f(y),y): y \in I\}$, where $I$ is the first/leftmost connected component of $S$
and the air bubble boundary
 $\B := \{ (f(y),y): y \in S\setminus I\}$.
\end{definition}

Note that as $u(x,0)=1$ and $u(x,A/B)=0$, each component of $\{ u> 0\}$ must be connected to the bottom by the maximum principle
applied to the subharmonic function $u$. So the set $\{ u> 0\}$ is connected, and the function $f$ must be positive at each point of definition. Moreover $\S$ has to be connected to $x=1/2$ as it can be the above discussion not have more than $1$ intersection
point with $x=0$ and it cannot be connected to the bottom. In other words, $\text{Im}(f|I) = [0,1/2]$.
 
\begin{lemma}[Non-degeneracy on $\S\setminus (0,A/B)$]\label{lem:ndeg}
	Let $u$ be a symmetrically decreasing solution of Lemma \ref{lem:symsol}
	and let $(x_0,y_0)\in \S\setminus (0,A/B)$. Then $u$ is non-degenerate at
	 $(x_0,y_0)$, that is,
	 $$ 	\liminf_{r\to 0} r^{-4} \int_{B_r(x_0,y_0)} u^2 > 0.$$
\end{lemma}
\begin{proof}
Suppose towards a contradiction that, setting $u_r(x,y) := u(x_0+rx,y_0+ry)/r$,
$$ 0= \liminf_{r\to 0} r^{-4} \int_{B_r(x_0,y_0)} u^2 = \liminf_{r\to 0}\int_{B_1(0)} u_r^2.$$
Then by Remark \ref{r:compactness}, there is a sequence $r_k\to 0$ such that for each smooth cut-off function $\eta \geq 0$
satisfying $\eta(x_0,y_0) > 0$,
$$ 0 =  \lim_{k\to\infty} \int_{B_1(0)}u_r \Delta \eta = \lim_{k\to\infty} \int_{B_1(0)}\eta \Delta u_r
= \lim_{k\to\infty} \int_{B_1(0)} \eta \sqrt{A-B (y_0+r_ky)} \h1 \lfloor \S_{r_k},$$
where $\S_r := \{(x_0+rx,y_0+ry) : (x,y)\in \S\} $.
Since $A-B y_0>0$, we obtain
$$ 0 = \lim_{k\to\infty} \int_{B_1(0)} \eta \h1 \lfloor \S_{r_k}.$$
On the other hand, by the graph property of $\S$,
$$ \lim_{k\to\infty} \int_{B_1(0)} \eta \h1 \lfloor \S_{r_k} \geq \eta(x_0,y_0),$$
a contradiction.
\end{proof}

\begin{lemma}[Blow-up limits]\label{lem:blowup}
Let $u$ be a symmetrically decreasing solution of Lemma \ref{lem:symsol}
	let $(x_0,y_0)\in \S\setminus (0,A/B)$
	and suppose that
	$u_r(x,y) := u(x_0+rx,y_0+ry)/r \to u_0$ weakly in $W^{1,2}_{loc}(\R^2)$.
	Then the following holds:
	\begin{enumerate}
\item If $0< x_0 < 1/2$, then either every blow-up limit at $(x_0,y_0)$ is of the form 
$u_0(x,y) = (A-By_0)\max((x,y)\cdot e,0)$, where $e$ is a unit vector satisfying $e_1 \leq 0$,
	or every blow-up limit at $(x_0,y_0)$ is of the form $u_0(x,y) = \theta |y|$ with $\theta \in (0,A-By_0]$.
	In the latter case only cusps pointing towards $x=0$ are allowed.
\item If $x_0=0$, then $u_0(x,y) = (A-By_0)\max(-y,0)$, and this is the unique blow-up limit.
\item If $x_0=1/2$, then either the unique blow-up limit is
	$u_0(x,y) = (A-By_0)\max(-y,0)$, or each blow-up limit is of the form $u_0(x,y) = \theta |x|$ with $\theta \in (0,A-By_0]$.	
\end{enumerate}
\end{lemma}
\begin{proof}
Most assertions follow from \cite[Proposition 4.7 (i)]{vw} together with Lemma \ref{lem:ndeg}
as well as the graph assumption,
observing that the 
perimeter assumption in \cite[Proposition 4.7]{vw} is satisfied by every variational
solution (see Remark \ref{r:compactness} and \cite[Lemma 3.3]{kw}).

Mixed asymptotics, that is some blow-up sequence converging to a half-plane solution
and another converging to an "absolute-value-solution" is not possible as we
would obtain by a continuity argument an intermediate sequence of radii
such that the limit would be neither a half-plane solution nor an "absolute-value-solution",
which is not possible.

In the case of a cusp pointing away from $x=0$, we find $y_0-\delta <y_l < y_0 < y_r< y_0+\delta$
such that on the interval $(y_l,y_r)$, the graph of $f$ is above the straight line segment
connecting $(y_l,f(y_l))$ and $(y_r,f(y_r))$.
Let us call the enclosed region $D$.
Then by Remark \ref{r:compactness} and Lemma \ref{lem:bern}, 
\begin{align*} 
0 & = \int_D \Delta u = \int_{\partial D \setminus \text{graph}(f)} \nabla u \cdot \nu \h1
+ \int_{\partial D \cap \text{graph}(f)} \nabla u \cdot \nu \h1\\
& \leq \int_{\partial D \setminus \text{graph}(f)}  \sqrt{A-By} \h1 -\int_{\partial D \cap \text{graph}(f)} \sqrt{A-By} \h1\\
& \leq o(\mathcal{H}^1(\partial D \cap \text{graph}(f))) + \sqrt{A-By_0}  \left(\int_{\partial D \setminus \text{graph}(f)}  \h1 -\int_{\partial D \cap \text{graph}(f)}  \h1\right),
\end{align*}
a contradiction as the straight line segment length is in this cusp case
$<  \mathcal{H}^1(\partial D \cap \text{graph}(f))/2$ provided that $\delta$ has been chosen small enough.
For this reason, combined with the assumption that $u$ is symmetrically decreasing,
at $x_0=1/2$, limits of the form $u_0(x,y) = \theta |y|$ are not possible either.
By the same reason, an upward pointing cusp of $\S$ at $x=0$ cannot happen.

Downward pointing/double cusps of $\S$ at $x=0$ are excluded by the fact that $f>0$.

Last, by the symmetry as well as periodicity of $u$, 
only horizontal and vertical asymptotics is allowed at $x=0$ and $x=1/2$.
\end{proof}

\begin{remark}
The $\theta$ in the previous lemma and thus the blow-up limit is unique
by Remark \ref{r:compactness} (similar to \cite[Lemma 6.4 (ii)]{vw}).
\end{remark}

\begin{corollary}[Flatness]\label{lem:flat}
Let $u$ be a symmetrically decreasing solution of Lemma \ref{lem:symsol},
	let $x_0<1/2$, $(x_0,y_0)\in \S\setminus (0,A/B)$
	and suppose that
	$u_r(x,y) := u(x_0+rx,y_0+ry)/r \to u_0$ weakly in $W^{1,2}_{loc}(\R^2)$.
	Then every blow-up limit is of the form 
$u_0(x,y) = (A-By_0)\max((x,y)\cdot e,0)$, where $e$ is a unit vector satisfying $e_1 \leq 0$.
\end{corollary}
\begin{proof}
In case of a cusps pointing towards $x=0$, 
$u$ is harmonic outside a cone of arbitrarily small angle, 
so we can apply Oddson's theorem \cite{oddson}
to obtain that 
$$ \sup_{B_r(x_0,y_0)} u \geq r^\mu \text{ for each } \mu > 1/2 \text{ and all }r < r_0 < 1,$$
contradicting the Lipschitz regularity of $u$.
\end{proof}

\begin{corollary}[Regularity]\label{lem:regular}
Let $u$ be a symmetrically decreasing solution of Lemma \ref{lem:symsol}.
Then $\S \setminus \left((0,A/B) \cup \{ |x| = 1/2 \}\right)$ is locally the graph of an analytic function.
Moreover, either $\S \setminus (0,A/B)$ is locally the graph of an analytic function,
or there is a downward-pointing cusp of $\S$ at $|x|=1/2$ at which non-$\S$ free boundary
points must exist that converge to the cusp point.
\end{corollary}
\begin{proof}
The analyticity follows from \cite[Theorem 8.1 and Theorem 8.4]{ac}
(note that $u$ is a weak solution by Lemma \ref{lem:ndeg} as well as Remark \ref{r:compactness}
while the flatness assumption follows from Lemma \ref{lem:flat}).
In the case of an upward-pointing cusp at $x=1/2$,  $u$ is harmonic 
outside a cone of angle $< \pi$ touching the cusp point, so we obtain as in the previous
proof a contradiction to the Lipschitz regularity of $u$.
\begin{remark}
It is not so surprising that singularities may appear on $x=1/2$ as 
by construction $x\to u(x,y)$ is decreasing on $(0,1/2)$ but not
in any open neighborhood of $x=1/2$, so we lose the monotonicity and the $y$-graph property in that neighborhood.
\end{remark}
\end{proof}

\bibliographystyle{plain}
\bibliography{ww_variational_g}

\section*{Acknowledgments}

We would like thank Antoine Laurain and Josue Daniel Diaz Avalos for performing a numerical analysis of the double constraint approach discussed in the introduction. DK was supported by NSF DMS grant 2247096. 

\end{document}